\documentclass[12pt,bezier]{article}
\usepackage{times}
\usepackage{booktabs}
\usepackage{pifont}
\usepackage{floatrow}
\usepackage{caption}
\usepackage{mathrsfs}
\usepackage[fleqn]{amsmath}
\usepackage{amsfonts,amsthm,amssymb,mathrsfs,bbding}
\usepackage{txfonts}
\usepackage{graphics,multicol}
\usepackage{graphicx}
\usepackage{color}
\usepackage{amssymb}
\usepackage{caption}
\usepackage{cite}
\usepackage{latexsym,bm}
\usepackage{indentfirst}
\usepackage{color}
\usepackage[colorlinks=true,anchorcolor=blue,filecolor=blue,linkcolor=blue,urlcolor=blue,citecolor=blue]{hyperref}
\usepackage{extarrows}
\usepackage{cite}
\usepackage{latexsym,bm}
\usepackage{mathtools}
\evensidemargin=\oddsidemargin\topmargin=-1.5cm

\newtheorem{thm}{Theorem}[section]
\newtheorem{prob}{Problem}[section]
\newtheorem{claim}{Claim}

\newtheorem{lem}{Lemma}[section]
\newtheorem{cor}{Corollary}[section]

\newtheorem{fact}{Fact}[section]

\theoremstyle{definition}

\addtocounter{section}{0}

\begin{document}
\title{Sufficient conditions for $k$-factors and spanning trees of graphs\footnote{Supported by National Natural Science Foundation of China
{(Nos. 11971445, 12371361 and 12171440)}.}}
\author{{\bf Guoyan Ao$^{a, b}$}, {\bf Ruifang Liu$^{a}$}\thanks{Corresponding author.
E-mail addresses: aoguoyan@163.com (G. Ao), rfliu@zzu.edu.cn (R. Liu), yuanjj@zzu.edu.cn (J. Yuan), daniel.ng@polyu.edu.hk (C.T. Ng), edwin.cheng@polyu.edu.hk (T.C.E. Cheng).},
{\bf Jinjiang Yuan$^{a}$}, {\bf C.T. Ng$^{c}$}, {\bf T.C.E. Cheng$^{c}$}\\
{\footnotesize $^a$ School of Mathematics and Statistics, Zhengzhou University, Zhengzhou, Henan 450001, China} \\
{\footnotesize $^b$ School of Mathematics and Physics, Hulunbuir University, Hailar, Inner Mongolia 021008, China}\\
{\footnotesize $^c$ Logistics Research Centre, Department of Logistics and Maritime Studies, }\\
{\footnotesize The Hong Kong Polytechnic University, Hong Kong SAR, People's Republic of China}}
\date{}

\date{}
\maketitle
{\flushleft\large\bf Abstract}
For any integer $k\geq1,$ a graph $G$ has a $k$-factor if it contains a $k$-regular spanning subgraph.
In this paper we prove a sufficient condition in terms of the number of $r$-cliques to guarantee the existence of a $k$-factor in a graph
with minimum degree at least $\delta$, which improves the sufficient condition of O \cite{O2021} based on the number of edges.
For any integer $k\geq2,$ a spanning $k$-tree of a connected graph $G$ is a spanning tree in which every vertex has degree at most $k$.
Motivated by the technique of Li and Ning \cite{Li2016}, we present a tight spectral condition for an $m$-connected graph to have a spanning $k$-tree,
which extends the result of Fan, Goryainov, Huang and Lin \cite{Fan2021} from $m=1$ to general $m$.
Let $T$ be a spanning tree of a connected graph. The leaf degree of $T$ is the maximum number of leaves adjacent to $v$ in $T$ for any $v\in V(T)$.
We provide a tight spectral condition for the existence of a spanning tree with leaf degree at most $k$
in a connected graph with minimum degree $\delta$, where $k\geq1$ is an integer.

\begin{flushleft}

\end{flushleft}

\section{Introduction}

Let $G$ be a graph with vertex set $V(G)$ and edge set $E(G)$. The order and size of $G$ are denoted by $|V(G)|=n$ and $|E(G)|=e(G)$, respectively.
We denote by $d_{G}(v)$, $\delta(G)$, $\omega(G)$ and $i(G)$ the degree of a vertex $v$ in $G$, the minimum degree, the clique number and
the number of isolated vertices of $G,$ respectively. The number of cliques of size $r$ in $G$ is denoted by $N_{r}(G)$.
We use $K_{n}$, $I_{n}$, and $R(n,t)$ to denote the complete graph of order $n$, the complement of $K_{n}$, and an $t$-regular graph with $n$ vertices.
For a vertex subset $S$ of $G$, we denote by $G-S$ and $G[S]$ the subgraph of $G$ obtained from $G$ by deleting the vertices in $S$ together with their incident edges and the subgraph of $G$ induced by $S$, respectively.
For two disjoint vertex subsets $X,Y \subset V(G)$, we denote by $G[X,Y]$ the induced bipartite subgraph between $X$ and $Y$.
Let $G_{1}$ and $G_{2}$ be two vertex-disjoint graphs. We denote by $G_{1}+G_{2}$ the disjoint union of $G_{1}$ and $G_{2}.$
The join $G_{1}\vee G_{2}$ is the graph obtained from $G_{1}+G_{2}$ by adding all the possible edges between $V(G_1)$ and $V(G_2)$.
Let $A(G)$ be the adjacency matrix of $G$. The largest eigenvalue of $A(G)$, denoted by $\rho(G)$, is called the {\it spectral radius} of $G$.

Let $P$ be a property defined on all the graphs of order $n$, and let $l$ be a positive integer.
We say $P$ is {\it$l$-stable} if whenever $G+uv$ has the property $P$ with $d_{G}(u)+d_{G}(v)\geq l$, then $G$ itself has the property $P$.
The {\it $l$-closure} of a graph $G$ is the graph obtained from $G$ by successively joining pairs of nonadjacent vertices
whose degree sum is at least $l$ until no such pair exists. Denote by $C_{l}(G)$ the $l$-closure of $G.$

F\"{u}redi, Kostochka, and Luo\cite{Furedi2019} posed the following open problem.

\begin{prob}\label{prob1}
Prove a sufficient condition in terms of the number of $r$-cliques to guarantee a graph $G$ to have an $l$-stable property $P.$
\end{prob}

A {\it linear forest} is a forest whose connected components are all the paths and isolated vertices.
A graph $G$ is called {\it $k$-hamiltonian} if each linear forest $F$ with $k$ edges contained in $G$ can be extended to a Hamilton cycle of $G.$
In the same paper, F\"{u}redi, Kostochka and Luo answered Problem 1.1 to assure that $G$ is $k$-hamiltonian $(l=n+k).$
Moreover, they suggested some $l$-stable properties as candidates, e.g., $G$ contains a cycle $C_{k}~(l=2n-k)$,
$G$ contains a path $P_{k}~(l=n-1)$, $G$ contains a matching $kK_{2}~(l=2k-1)$, $G$ contains a $1$-factor $(l=n-1)$, $G$ contains a $k$-factor for $k\geq2~(l=n+2k-4),$
$G$ is $k$-connected $(l=n+k-2)$, and $G$ is $k$-wise hamiltonian, i.e., every $n-k$ vertices span a $C_{n-k}$~$(l=n+k-2)$.

The corresponding Problem \ref{prob1} for the property containing Hamilton cycles,
long cycles or long paths have been well studied in \cite{Furedi2018, Luo2018, Ning2020}.
Duan et al. \cite{Duan2020} answered Problem \ref{prob1} for graphs containing a matching $kK_{2}.$
Let $\mathcal{L}_{k}$ be the family of all the linear forests of
size $k$ without isolated vertices. Xue et al. \cite{Xue2022} answered Problem \ref{prob1} to assure that $G$ is $\mathcal{L}_{k}$-free $(l=k).$
In this paper we focus on Problem \ref{prob1} for graphs containing a $1$-factor and a $k$-factor for $k\geq2,$ respectively.

For any integer $k\geq 1,$ a {\it$k$-factor} of $G$ is a $k$-regular spanning subgraph.
It is well known that the complete graph $K_n$ has a $k$-factor when $nk$ is even.
Bondy and Chv\'{a}tal \cite{Bondy1976} presented a closure theorem to assure that a graph has a $k$-factor.

\begin{thm}[Bondy and Chv\'{a}tal \cite{Bondy1976}]\label{th1}
Let $G$ be a graph of order $n$ and let $1\leq k<n$ be an integer. Then\\
(i) $G$ has a $1$-factor if and only if $C_{n-1}(G)$ has a $1$-factor. \\
(ii) $G$ has a $k$-factor if and only if $C_{n+2k-4}(G)$ has a $k$-factor for $k\geq 2.$
\end{thm}

There are many sufficient conditions to assure that a graph contains a $k$-factor (see e.g.,
\cite{Fan2022, Kouider2004, Li1998, Li2003, O2022}). Inspired by the related work on Problem \ref{prob1},
we prove a sufficient condition in terms of the number of $r$-cliques to guarantee the existence of a $k$-factor
in a graph with minimum degree at least $\delta$.
Recall that $I_{n}=nK_{1}$.

Let $n, r, k$ and $q$ be integers. We define two functions as follows:
$$\varphi(n,r,q)={n-q-1\choose r}+(q+1){q\choose r-1},$$
$$\phi(n,r,k,q)={n-q+2k-4\choose r}+(q-2k+4){q\choose r-1}.$$

\begin{thm}\label{main1}
Let $r, k$ and $\delta$ be integers with $r\geq 2$ and $k\geq 1$. Suppose that $G$ is a graph of order $n\geq 2\delta+k+1$ with $nk$ even and minimum degree at least $\delta$. Then the following two statements hold.\\
(i) If $N_{r}(G)>{\rm max} \left\{\varphi(n,r,\delta+1), \varphi(n,r,\frac{n}{2}-1)\right\}$ with $1\leq \delta \leq \frac{n}{2}-1$,
then $G$ has a $1$-factor unless
$C_{n-1}(G)\cong K_{n-\delta-1}\cup K_{\delta+1}$ ($\delta$ is even) or $K_{\delta}\vee(K_{n-2\delta-1}+I_{\delta+1})$.\\
(ii) If $N_{r}(G)>{\rm max} \left\{\phi(n,r,k,\delta+1), \phi(n,r,k,\lfloor\frac{n+2k-5}{2}\rfloor)\right\}$
with $2k-2\leq \delta \leq \lfloor\frac{n+2k-5}{2}\rfloor$, then $G$ has a $k$-factor for $k\geq 2$.
\end{thm}

By the definition of $\varphi(n,r,q)$, we know that Theorem \ref{main1} (i) improves the following result for a large enough $n$.

\begin{thm}[Duan et al. \cite{Duan2020}]\label{th22}
Let $G$ be a graph of order $n$ with minimum degree $\delta$.
If $N_{r}(G)>{\rm max} \left\{\varphi(n,r,\delta), \varphi(n,r,\frac{n}{2}-1)\right\}$ for each $r\geq2$, then $G$ has a $1$-factor.
\end{thm}

By direct calculation, $\varphi(n,2,\delta+1)\geq\varphi(n,2,\frac{n}{2}-1)$ for $n\geq 6\delta+10,$ and
$\phi(n,2,k,\delta+1)\geq \phi(n,2,k,\lfloor\frac{n+2k-5}{2}\rfloor)$ for $n\geq 6\delta+4k-3$.
Combining Theorem \ref{main1}, we immediately deduce the following result.

\begin{cor}\label{cor1}
Let $G$ be a graph of order $n$ with $nk$ even and minimum degree at least $\delta$.\\
(i) If $e(G)>\varphi(n,2,\delta+1)$ with $n\geq 6\delta+10$ and $1\leq \delta \leq \frac{n}{2}-1$,
then $G$ has a $1$-factor unless $C_{n-1}(G)\cong K_{n-\delta-1}\cup K_{\delta+1}$ ($\delta$ is even) or $K_{\delta}\vee(K_{n-2\delta-1}+I_{\delta+1})$.\\
(ii) If $e(G)>\phi(n,2,k,\delta+1)$ with $n\geq 6\delta+4k-3$ and $2k-2\leq \delta \leq \lfloor\frac{n+2k-5}{2}\rfloor$,
then $G$ has a $k$-factor for $k\geq 2$.
\end{cor}

O \cite{O2021}, Wei and Zhang \cite{Wei2023} presented sufficient conditions in terms of $e(G)$ to guarantee
a graph to have a $1$-factor and a $k$-factor, respectively.

\begin{thm}[O \cite{O2021}]\label{th2}
Let $G$ be a connected graph of order $n$, where $n\geq 10$ is even. If $e(G)> {n-2\choose 2}+2$, then $G$ has a $1$-factor.
\end{thm}

The lower bound in Corollary \ref{cor1} (i) is ${n-3\choose 2}+6$. Hence our result improves Theorem \ref{th2} for $n\geq 16$.

\begin{thm}[Wei and Zhang \cite{Wei2023}]\label{th3}
Let $G$ be a graph of order $n\geq k+1$ with $nk$ even and minimum degree $\delta\geq k\geq 1$.
If $e(G)> {n-1\choose 2}+\frac{k+1}{2}$, then $G$ has a $k$-factor.
\end{thm}

The lower bound in Corollary \ref{cor1} (ii) is ${n-3\choose 2}+6k-3$. Hence our result improves Theorem \ref{th3} for $n\geq 16k-15$.

\begin{cor}\label{main4}
Let $k$ and $\delta$ be integers with $k\geq 2$.
Suppose that $G$ is a graph of order $n$ with $nk$ even and minimum degree at least $\delta$. \\
(i) If $\rho(G)>\frac{\delta-1}{2}+\sqrt{n^{2}-(3\delta+5)n+\frac{13\delta^{2}+23\delta+41}{4}},$
$1\leq \delta \leq \frac{n}{2}-1$ and $n\geq 6\delta+10$, then $G$ has a $1$-factor unless
$C_{n-1}(G)\cong K_{n-\delta-1}\cup K_{\delta+1}$ ($\delta$ is even) or $K_{\delta}\vee(K_{n-2\delta-1}+I_{\delta+1})$.\\
(ii) If $\rho(G)>\frac{\delta-1}{2}+\sqrt{n^{2}-(3\delta-4k+11)n+\frac{13\delta^{2}-(32k-94)\delta+16k^{2}-104k+161}{4}},$
 $2k-2\leq \delta \leq \lfloor\frac{n+2k-5}{2}\rfloor$ and $n\geq 6\delta+4k-3$, then $G$ has a $k$-factor for $k\geq 2$.
\end{cor}

For any integer $k\geq2,$ a {\it spanning $k$-tree} of a connected graph $G$ is a spanning tree in which every vertex has degree at most $k$.
Note that a spanning $2$-tree is a Hamilton path. Hence a spanning $k$-tree of a connected graph is a natural generalization of a Hamilton path.
Ozeki and Yamashita \cite{Ozeki2011} proved that it is an $\mathcal{NP}$-complete problem to decide
whether a given connected graph admits a spanning $k$-tree.
Win \cite{Win1989} proved a toughness-type condition for the existence of a spanning $k$-tree in a connected graph.
Ellingham and Zha \cite{Ellingham2000} presented a short proof to Win's theorem.
Using the toughness-type condition, Fan et al. \cite{Fan2021} posed a spectral condition for
the existence of a spanning $k$-tree in a connected graph.

\begin{thm}[Fan et al. \cite{Fan2021}]\label{th4}
Let $G$ be a connected graph of order $n\geq 2k+16$, where $k\geq3$ is an integer. If
$\rho(G)\geq \rho(K_{1}\vee (K_{n-k-1}+I_{k})),$
then $G$ has a spanning $k$-tree unless $G\cong K_{1}\vee(K_{n-k-1}+I_{k})$.
\end{thm}

Kano and Kishimoto \cite{Kano2011} presented a closure theorem to assure that an $m$-connected graph has a spanning $k$-tree.

\begin{thm}[Kano and Kishimoto \cite{Kano2011}]\label{th5}
Let $G$ be an $m$-connected graph of order $n\geq k+1$, where $m\geq1$ and $k\geq 2$ are integers.
Then $G$ has a spanning $k$-tree if and only if the $(n-km+2m-1)$-closure $C_{n-(k-2)m-1}(G)$ of $G$ has a spanning $k$-tree.
\end{thm}

For more results on spanning $k$-tree, the reader can refer to \cite{Nakamoto2009, Thomassen1994, Kawarabayashi2003, Neumann1991}.
Motivated by the ideas on Hamilton path from Li and Ning \cite{Li2016} and using typical spectral techniques,
we prove a tight spectral condition to guarantee the existence of a spanning $k$-tree in an $m$-connected graph.

\begin{thm}\label{main2}
Let $G$ be an $m$-connected graph of order $n\geq {\rm max}\{(7k-2)m+4, (k-1)m^{2}+\frac{1}{2}(3k+1)m+\frac{9}{2}\}$,
where $m\geq 1$ and $k\geq2$ are integers. If $\rho(G)\geq \rho(K_{m}\vee (K_{n-km-1}+I_{km-m+1})),$
then $G$ has a spanning $k$-tree unless $G\cong K_{m}\vee (K_{n-km-1}+I_{km-m+1})$.
\end{thm}

Kaneko \cite{Kaneko2001} introduced the concept of leaf degree of a spanning tree.
Let $T$ be a spanning tree of a connected graph $G$.
The {\it leaf degree of a vertex $v\in V(T)$} is defined as the number of leaves adjacent to vertex $v$ in $T$.
If $v$ is a leaf of $T$ with at least three vertices, then the leaf degree of $v$ equals zero.
Furthermore, the {\it leaf degree of $T$} is the maximum leaf degree among all the vertices of $T$.
They gave a necessary and sufficient condition for a connected graph to contain a spanning tree with leaf degree at most $k$.

\begin{thm}[Kaneko \cite{Kaneko2001}]\label{th6}
Let $G$ be a connected graph and $k\geq 1$ be an integer. Then $G$ has a spanning tree with leaf degree at most $k$ if and only if
$i(G-S)<(k+1)|S|$ for every nonempty subset $S\subseteq V(G)$.
\end{thm}

Using Kaneko's theorem, Ao, Liu and Yuan \cite{Ao2023} presented a tight spectral condition
for the existence of a spanning tree with leaf degree at most $k$ in a connected graph.

\begin{thm}[Ao, Liu and Yuan \cite{Ao2023}]\label{th7}
Let $G$ be a connected graph of order $n\geq 2k+12$, where $k\geq1$ is an integer.
If $\rho(G)\geq \rho(K_{1}\vee (K_{n-k-2}+I_{k+1})),$
then $G$ has a spanning tree with leaf degree at most $k$ unless $G\cong K_{1}\vee (K_{n-k-2}+I_{k+1})$.
\end{thm}

Motivated by Theorem \ref{th7}, we prove a tight spectral condition
for the existence of a spanning tree with leaf degree at most $k$ in a connected graph with minimum degree $\delta$.

\begin{thm}\label{main3}
Let $G$ be a connected graph of order $n\geq 3(k+2)\delta+2$ with minimum degree $\delta$,
where $k\geq1$ is an integer. If
$\rho(G)\geq \rho(K_{\delta}\vee(K_{n-k\delta-2\delta}+ I_{k\delta+\delta})),$
then $G$ has a spanning tree with leaf degree at most $k$ unless $G\cong K_{\delta}\vee(K_{n-k\delta-2\delta}+ I_{k\delta+\delta})$.
\end{thm}

\section{Proof of Theorem \ref{main1}}

In this section we prove Theorem \ref{main1} by using closure theory. Let $G$ be a graph on $n$ vertices.
If there are at least $s$ vertices in $V(G)$ with degree at most $q$, then we say $G$ has {\it$(s,q)$-P\'{o}sa property}.
F\"{u}redi, Kostochka, and Luo\cite{Furedi2019} established the following fact.

\begin{fact}[F\"{u}redi, Kostochka, and Luo\cite{Furedi2019}]\label{fac1}
{\rm If $G$ has $(s,q)$-P\'{o}sa property and $n\geq s+q$, then $$N_{r}(G)\leq {n-s\choose r}+s{q\choose r-1}.$$ }
\end{fact}

\begin{lem}[Xue, Liu and Kang \cite{Xue2022}]\label{le1}
Let property $P$ be $l$-stable and the complete graph $K_n$ has the property $P$.
Suppose that $G$ is a graph of order $n$ with minimum degree at least $\delta$.
If $G$ does not have the property $P$, then there exists an integer $q$ with $\delta\leq q\leq \lfloor\frac{l-1}{2}\rfloor$
such that $G$ has $(n-l+q,q)$-P\'{o}sa property.
\end{lem}

The following lemma gives a structural characterization of graphs with P\'{o}sa property.

\begin{lem}[Duan et al. \cite{Duan2020}]\label{le2}
Suppose that $G$ has $n$ vertices and is stable under taking $l$-closure.
Let $q$ be the maximum integer such that $G$ has $(n-l+q,q)$-P\'{o}sa property and $q\leq \lfloor\frac{l-1}{2}\rfloor$.
If $U$ is the set of vertices in $V(G)$ with degree greater than $q$, then $G[U]$ is a complete graph.
\end{lem}

\begin{lem}[Dirac \cite{Dirac1952}]\label{le3}
Let $G$ be a graph of order $n$ with minimum degree $\delta$, where $\delta\geq\frac{n}{2}$ and $n\geq 3$. Then $G$ is hamiltonian.
\end{lem}

\begin{lem}[Katerinis \cite{Katerinis1985}]\label{le4}
Let $G$ be a graph of order $n$ with $n\geq 4k-5$, and let $k$ be a positive integer and $nk$ be even.
If $\delta(G)\geq\frac{n}{2}$, then $G$ has a $k$-factor.
\end{lem}

Now we are in a position to present the proof of Theorem \ref{main1}. Recall that $R(n,t)$ is a $t$-regular graph on $n$ vertices.

\medskip
\noindent  \textbf{Proof of Theorem \ref{main1}.}
(i) Assume to the contrary that $G$ has no $1$-factor. Let $H=C_{n-1}(G)$. It suffices to prove that $H\cong K_{n-\delta-1}\cup K_{\delta+1}$ ($\delta$ is even) or $K_{\delta}\vee(K_{n-2\delta-1}+I_{\delta+1})$.

By Theorem \ref{th1} (i), $H$ has no $1$-factor.
By Lemma \ref{le1}, there exists an integer $q$ with $\delta\leq q\leq \frac{n}{2}-1$ such that $H$ has
$(q+1,q)$-P\'{o}sa property. Let $q$ be the maximum integer with the above P\'{o}sa property. First, we prove the following claim.

\begin{claim}\label{cla11}
{\rm $q=\delta$.}
\end{claim}

\begin{proof}
Suppose that $\delta+1\leq q\leq \frac{n}{2}-1$. By Fact \ref{fac1}, we have
$$N_{r}(H)\leq {n-q-1\choose r}+(q+1){q\choose r-1}=\varphi(n,r,q).$$
Since $G\subseteq H$, we have $N_{r}(G)\leq\varphi(n,r,q)$. So
$N_{r}(G)\leq{\rm max} \left\{\varphi(n,r,\delta+1), \varphi(n,r,\frac{n}{2}-1)\right\},$
which contradicts the assumption.
\end{proof}

Note that the minimum degree of $G$ is at least $\delta$.
Combining Claim \ref{cla11}, P\'{o}sa property of $H$, and the maximality of $q$,
we know that there are exactly $\delta+1$ vertices of degree $\delta$ in $H$.
Let $X$ be the set of vertices with degree $\delta$ in $H$ and $C=V(H)\setminus X$.
Then $|X|=\delta+1$ and $|C|=n-\delta-1$. Note that the degrees of the vertices in $C$ are greater than $\delta$. By Lemma \ref{le2}, $C$ forms a clique in $H$.

Let $Y=\{v: d_{H}(v)\geq n-\delta-1\}$. Note that $\delta \leq \frac{n}{2}-1$. Then $Y\subseteq C.$
For $u\in X$ and $v\in Y$, we have $d_{H}(u)+d_{H}(v)\geq \delta+(n-\delta-1)=n-1$.
Note that $H$ is an $(n-1)$-closed graph.
Then every vertex of $Y$ is adjacent to all the vertices of $X$, so $H[X,Y]$ forms a complete bipartite graph.
By the definition of $Y,$ we have $d_{H}(v)=n-\delta-2$ for each $v\in C\backslash Y.$
Note that $|C|=n-\delta-1$. Hence $N_{H}(u)\subseteq X\cup Y$ for each $u\in X$.

\begin{claim}\label{cla112}
{\rm $0\leq|Y|\leq \delta$.}
\end{claim}

\begin{proof}
If $|Y|\geq\delta+1$, then $d_{H}(u)\geq \delta+1$ for any $u\in X$, a contradiction.
\end{proof}

Based on Claim \ref{cla112}, we distinguish the following three cases.

\vspace{1.5mm}
\noindent\textbf{Case 1.} $|Y|=0$.
\vspace{1.5mm}

Recall that $|X|=\delta+1$ and $d_{H}(u)=\delta$ for each $u\in X$.
It follows that $H[X]$ forms a complete graph. Hence $H\cong K_{n-\delta-1}\cup K_{\delta+1}$. Note that $n$ is even.
If $\delta$ is odd, then $K_{n-\delta-1}\cup K_{\delta+1}$ has a $1$-factor, a contradiction.
If $\delta$ is even, then $K_{n-\delta-1}\cup K_{\delta+1}$ has no $1$-factor.
Hence $H\cong K_{n-\delta-1}\cup K_{\delta+1}$.

\vspace{1.5mm}
\noindent\textbf{Case 2.} $|Y|=\delta$.
\vspace{1.5mm}

Note that $H[X,Y]$ forms a complete bipartite graph and $d_{H}(u)=\delta$ for $u\in X$. Then $X$ is an independent set.
This means that $H\cong K_{\delta}\vee(K_{n-2\delta-1}+I_{\delta+1})$.
Note that the vertices of $I_{\delta+1}$ are only adjacent to all the vertices of $K_{\delta}$.
It is easy to see that $K_{\delta}\vee(K_{n-2\delta-1}+I_{\delta+1})$ has no $1$-factor.
Hence $H\cong K_{\delta}\vee(K_{n-2\delta-1}+I_{\delta+1})$.

\vspace{1.5mm}
\noindent\textbf{Case 3.} $1\leq|Y|\leq \delta-1$.
\vspace{1.5mm}

Then $\delta\geq 2$. Let $|Y|=s$.
Recall that $H[X,Y]$ forms a complete bipartite graph. Hence $H\cong K_{s}\vee(K_{n-s-\delta-1}+R(\delta+1,\delta-s))$.
Let $F=H[X\cup Y]$. Then $\delta+2 \leq|V(F)|\leq 2\delta$, and hence $\delta\geq \frac{|V(F)|}{2}$.
By Lemma \ref{le3}, $F$ has a Hamilton cycle.
It follows that $H$ has a Hamilton path. Note that $n$ is even. Then $H$ has a $1$-factor, a contradiction.

\vspace{1.5mm}

(ii) Suppose that $G$ has no $k$-factor for $k\geq2$. Let $H=C_{n+2k-4}(G)$. By Theorem \ref{th1} (ii), $H$ has no $k$-factor.
According to Lemma \ref{le1}, there exists an integer $q$ with $\delta\leq q\leq \lfloor\frac{n+2k-5}{2}\rfloor$
such that $H$ has $(q-2k+4,q)$-P\'{o}sa property.
Let $q$ be the maximum one with the above P\'{o}sa property. We prove the following claim.

\begin{claim}\label{cla12}
{\rm $q=\delta$.}
\end{claim}

\begin{proof}
Assume that $\delta+1\leq q\leq \lfloor\frac{n+2k-5}{2}\rfloor$. By Fact \ref{fac1}, we have
$$N_{r}(H)\leq {n-q+2k-4\choose r}+(q-2k+4){q\choose r-1}=\phi(n,r,k,q).$$
Note that $G\subseteq H$. It follows that $N_{r}(G)\leq{\rm max} \left\{\phi(n,r,k,\delta+1), \phi(n,r,k,\lfloor\frac{n+2k-5}{2}\rfloor)\right\},$
which contradicts the assumption.
\end{proof}

By Claim \ref{cla12}, P\'{o}sa property of $H$, and the maximality of $q$, there are exactly $\delta-2k+4$ vertices of degree $\delta$ in $H$.
Let $X$ be the set of vertices with degree $\delta$ in $H$ and $C=V(H)\setminus X$.
Then $|X|=\delta-2k+4$ and $|C|=n-\delta+2k-4$. By Lemma \ref{le2}, $C$ forms a clique in $H$.

Let $Y=\{v: d_{H}(v)\geq n-\delta+2k-4\}$. Note that $\delta \leq \lfloor\frac{n+2k-5}{2}\rfloor$.
Then we have $n-\delta+2k-4\geq \delta+1$, so $Y\subseteq C.$
For $u\in X$ and $v\in Y$, we have $d_{H}(u)+d_{H}(v)\geq \delta+(n-\delta+2k-4)=n+2k-4$.
Note that $H$ is an $(n+2k-4)$-closed graph.
Then every vertex of $Y$ is adjacent to all the vertices of $X$, so $H[X,Y]$ forms a complete bipartite graph.
Note that $|C|=n-\delta+2k-4$ and $C$ is a clique. Combining the definition of $Y,$ we have $d_{H}(v)=n-\delta+2k-5$ for each $v\in C\backslash Y.$
Hence $N_{H}(u)\subseteq X\cup Y$ for each $u\in X$.

\begin{claim}\label{cla13}
{\rm $2k-3\leq|Y|\leq \delta$.}
\end{claim}

\begin{proof}
In fact, if $|Y|\geq\delta+1$, then $d_{H}(u)\geq \delta+1$ for each $u\in X$, a contradiction.
Note that $d_{H}(u)=\delta$ and $N_{H}(u)\subseteq X\cup Y$ for each $u\in X$.
Then $|X\cup Y|\geq \delta+1$, so $|Y|\geq\delta+1-|X|=\delta+1-(\delta-2k+4)=2k-3.$
Therefore, $2k-3\leq|Y|\leq \delta$.
\end{proof}

Let $|Y|=s$. By Claim \ref{cla13}, we have $2k-3\leq s\leq \delta$. Note that $k\geq 2$.

\vspace{1.5mm}
\noindent\textbf{Case 1.} $k=2$.
\vspace{1.5mm}

Then $1\leq s\leq \delta$ and $|X|=\delta-2k+4=\delta\geq 2k-2=2$.
Obviously, $H\cong K_{s}\vee (K_{n-s-\delta}+R(\delta,\delta-s))$.
Let $F=H[X\cup Y]$. Then $F\cong K_{s}\vee R(\delta,\delta-s)$ and $|V(F)|=\delta+s\geq 3$.
Clearly, $F\cup K_{n-s-\delta}$ is a spanning subgraph of $H$.
Note that
$$\delta(F)=\min\{\delta, \delta+s-1\}=\delta=\frac{2\delta}{2}\geq \frac{\delta+s}{2}=\frac{|V(F)|}{2}.$$
By Lemma \ref{le3}, $F$ is hamiltonian, and hence $F$ has a $2$-factor.
Note that $n\geq 2\delta+k+1= 2\delta+3$. Then $n-s-\delta\geq (2\delta+3)-\delta-\delta=3,$ so $K_{n-\delta-s}$ has a $2$-factor.
Therefore, $F\cup K_{n-s-\delta}$ has a $2$-factor. It follows that $H$ has a $2$-factor, a contradiction.

\vspace{1.5mm}
\noindent\textbf{Case 2.} $k\geq 3$.
\vspace{1.5mm}

Note that $2k-3\leq s\leq \delta$. We divide the proof into two cases.

\vspace{1.5mm}
\noindent\textbf{Case 2.1.} $s=2k-3$.
\vspace{1.5mm}

Then $H\cong K_{2k-3}\vee (K_{n-\delta-1}+K_{\delta-2k+4})$.
Note that $K_{\delta+1}\cup K_{n-\delta-1}$ is a spanning subgraph of $H$.
If $\delta$ is odd, then $(\delta+1)k$ is even, and hence $K_{\delta+1}$ has a $k$-factor.
Since $n\geq 2\delta+k+1$, $n-\delta-1\geq (2\delta+k+1)-\delta-1=\delta+k> k+1$.
Note that $nk$ is even. Then $(n-\delta-1)k$ is even, and hence $K_{n-\delta-1}$ has a $k$-factor.
It follows that $K_{\delta+1}\cup K_{n-\delta-1}$ has a $k$-factor, so $H$ has a $k$-factor, a contradiction.

\begin{figure}
\centering
\includegraphics[width=0.95\textwidth]{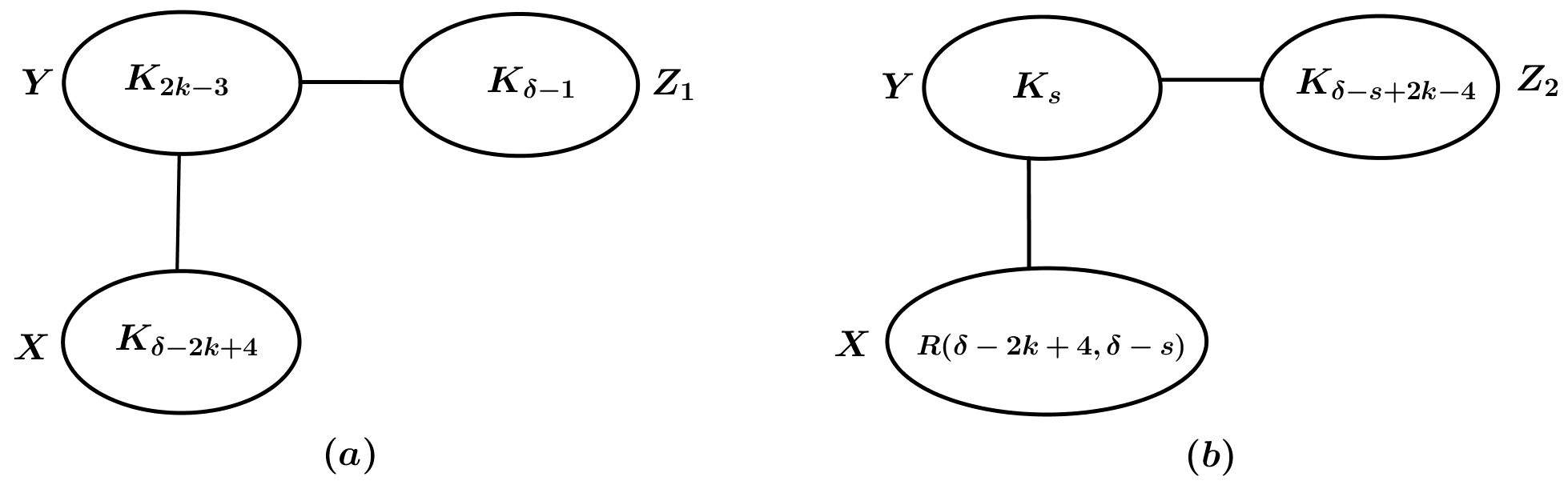}\\
\caption{Graphs $K_{2k-3}\vee (K_{\delta-1}+K_{\delta-2k+4})$ and $K_{s}\vee (K_{\delta-s+2k-4}+R(\delta-2k+4, \delta-s)).$
}\label{fig1}
\end{figure}

Next we consider $\delta$ is even. By Claim \ref{cla13}, we have $\delta\geq 2k-3$, and hence $\delta\geq 2k-2$.
Recall that $n-\delta-1\geq \delta+k>\delta-1$.
Now we construct a proper subgraph $L= H[X\cup Y\cup Z_{1}]$ with $2\delta$ vertices in $H$, where
$X=V(K_{\delta-2k+4})$, $Y=V(K_{2k-3})$ and $Z_{1}=V(K_{\delta-1})\subset V(K_{n-\delta-1})$.
Then $L\cong K_{2k-3}\vee (K_{\delta-1}+K_{\delta-2k+4})$ (see Fig. \ref{fig1}$(a)$).
Note that $L\cup K_{n-2\delta}$ is a spanning subgraph of $H$.
It is easy to see that $|V(L)|=2\delta\geq 2(2k-2)=4k-4$ and
$$\delta(L)=\min\{\delta, 2\delta-1, \delta+2k-5\}=\delta=\frac{|V(L)|}{2}.$$
Note that $|V(L)|k$ is even. By Lemma \ref{le4}, $L$ has a $k$-factor.
Recall that $n-2\delta\geq k+1$. Since $nk$ is even, $(n-2\delta)k$ is even. Then $K_{n-2\delta}$ has a $k$-factor,
so $L\cup K_{n-2\delta}$ has a $k$-factor. It follows that $H$ has a $k$-factor, a contradiction.

\vspace{1.5mm}
\noindent\textbf{Case 2.2.} $2k-2\leq s\leq \delta$.
\vspace{1.5mm}

Then $H\cong K_{s}\vee (K_{n-s-\delta+2k-4}+R(\delta-2k+4, \delta-s))$.
Note that $$n-s-\delta+2k-4\geq (2\delta+k+1)-s-\delta+2k-4=\delta-s+3k-3>\delta-s+2k-4.$$
Let $X=V(R(\delta-2k+4, \delta-s))$, $Y=V(K_{s})$, and $Z_2=V(K_{\delta-s+2k-4})\subset V(K_{n-s-\delta+2k-4})$.
Next we construct a proper subgraph $M=H[X\cup Y\cup Z_2]$ of $H$.
Then $|V(M)|=2\delta$ and $M\cong K_{s}\vee (K_{\delta-s+2k-4}+R(\delta-2k+4, \delta-s))$ (see Fig. \ref{fig1}$(b)$).
Obviously, $M\cup K_{n-2\delta}$ is a spanning subgraph of $H$ and $n-2\delta\geq k+1$.
Note that $|V(M)|=2\delta\geq 2(2k-2)=4k-4$ and
$$\delta(M)=\min\{\delta, 2\delta-1, \delta+2k-5\}=\delta=\frac{|V(M)|}{2}.$$
Clearly, $|V(M)|k$ is even. By Lemma \ref{le4}, $M$ has a $k$-factor. Recall that $K_{n-2\delta}$ has a $k$-factor.
Hence $M\cup K_{n-2\delta}$ has a $k$-factor.
It follows that $H$ has a $k$-factor, a contradiction. \hspace*{\fill}$\Box$

We use the following lemma to prove Corollary \ref{main4}.

\begin{lem}[Hong, Shu and Fang \cite{Hong2001}, Nikiforov \cite{Nikiforov2002}]\label{le41}
Let $G$ be a graph with minimum degree $\delta.$ Then
$$\rho(G)\leq \frac{\delta-1}{2}+\sqrt{2e(G)-\delta n+\frac{(\delta+1)^{2}}{4}}.$$
\end{lem}

\medskip
\noindent  \textbf{Proof of Corollary \ref{main4}.}
(i) By the assumption and Lemma \ref{le41}, we have
  $$\frac{\delta-1}{2}+\sqrt{n^{2}-(3\delta+5)n+\frac{13\delta^{2}+23\delta+41}{4}}<\rho(G)\leq\frac{\delta-1}{2}+\sqrt{2e(G)-\delta n+\frac{(\delta+1)^{2}}{4}}.$$
By direct calculation, we obtain
$$e(G)>\frac{n^{2}}{2}-(\delta+\frac{5}{2})n+\frac{3}{2}\delta^{2}+\frac{11}{2}\delta+5=\varphi(n,2,\delta+1).$$
By Corollary \ref{cor1}, $G$ has a $1$-factor unless
$C_{n-1}(G)\cong K_{n-\delta-1}\cup K_{\delta+1}$ ($\delta$ is even) or $K_{\delta}\vee(K_{n-2\delta-1}+I_{\delta+1})$.

\vspace{1.5mm}

(ii) By Lemma \ref{le41}, we have $$\rho(G)\leq\frac{\delta-1}{2}+\sqrt{2e(G)-\delta n+\frac{(\delta+1)^{2}}{4}}.$$
Note that $\rho(G)>\frac{\delta-1}{2}+\sqrt{n^{2}-(3\delta-4k+11)n+\frac{13\delta^{2}-(32k-94)\delta+16k^{2}-104k+161}{4}}.$ Then
$$e(G)>\frac{n^{2}}{2}-(\delta-2k+\frac{11}{2})n+\frac{3}{2}\delta^{2}-(4k-\frac{23}{2})\delta+2k^{2}-13k+20=\phi(n,2,k,\delta+1).$$
By Corollary \ref{cor1}, $G$ has a $k$-factor. \hspace*{\fill}$\Box$

\section{Proof of Theorem \ref{main2}}

Before providing our main theorem, we first show that an $m$-connected and $(n-km+2m-1)$-closed graph $G$ contains a large clique
if its number of edges is sufficiently large.
The technique for proving Lemma \ref{le5} is mainly motivated by that used in \cite{Li2016}.

\begin{lem}\label{le5}
Let $H$ be an $m$-connected and $(n-km+2m-1)$-closed graph of order $n\geq (7k-2)m+4$, where $m$ and $k$ are integers with $m\geq 1$ and $k\geq2$.
If $$e(H)\geq {n-(k-1)m-2\choose 2}+(k-1)m^{2}+(k+1)m+3,$$ then $\omega(H)\geq n-(k-1)m-1$.
\end{lem}

\begin{proof}
Suppose that $H$ is an $m$-connected and $(n-km+2m-1)$-closed graph of order $n\geq (7k-2)m+4$, where $m\geq 1$ and $k\geq2$.
Then any two vertices of degree at least $\frac{n-(k-2)m-1}{2}$ must be adjacent in $H$.
Let $C$ be the vertex set of a maximum clique of $H$ which contains all the vertices of degree at least $\frac{n-(k-2)m-1}{2},$
and let $F$ be the subgraph of $H$ induced by $V(H)\setminus C.$ Let $|C|=r$. Then $|V(F)|=n-r$.

\vspace{1.5mm}
\noindent\textbf{Case 1.} $1 \leq r\leq \frac{n}{3}+(k-1)m+1.$
\vspace{1.5mm}

Note that $e(H[C])={r\choose 2}$, $e(C, V(F))=\sum_{u\in V(F)}d_{C}(u)$,
and
 $$e(F)=\frac{\sum_{u\in V(F)}d_{H}(u)-\sum_{u\in V(F)}d_{C}(u)}{2}.$$
We first prove the following claim.

\begin{claim}\label{cla1}
{\rm $d_{H}(u)\leq \frac{n-(k-2)m-2}{2}$ and $d_{C}(u)\leq r-1$ for each $u\in V(F)$.}
\end{claim}

\begin{proof}
Assume that there exists a vertex $u\in V(F)$ such that $d_{H}(u)\geq\frac{n-(k-2)m-1}{2}$. Then $d_{H}(u)+d_{H}(v)\geq n-(k-2)m-1$ for each $v\in C$.
Note that $H$ is an $(n-km+2m-1)$-closed graph.
Then the vertex $u$ is adjacent to every vertex of $C,$ and hence $C\cup\{u\}$ is a larger clique, which contradicts the maximality of $C$.
Suppose that there exists a vertex $u\in V(F)$ such that $d_{C}(u)\geq r$. Note that $|C|=r$.
Then $d_{C}(u)=r,$ so $u$ is adjacent to each vertex of $C.$
It follows that $C\cup\{u\}$ is a larger clique, a contradiction.
\end{proof}

By Claim \ref{cla1}, we have
\begin{eqnarray*}
e(H)&=&e(H[C])+e(C,V(F))+e(F)\\
&=&{r\choose 2}+\frac{\sum_{u\in V(F)}d_{C}(u)+\sum_{u\in V(F)}d_{H}(u)}{2}\\
&\leq&{r\choose 2}+\frac{(r-1)(n-r)}{2}+\frac{(n-km+2m-2)(n-r)}{4}\\
&=&\frac{1}{4}(n+km-2m+2)r+\frac{n^{2}}{4}-\frac{1}{4}(km-2m+4)n\\
&\leq&\frac{1}{4}(n+km-2m+2)(\frac{n}{3}+km-m+1)+\frac{n^{2}}{4}-\frac{1}{4}(km-2m+4)n\\
&=&\frac{n^{2}}{3}+\frac{1}{12}(km+m-7)n+\frac{1}{4}(k^{2}-3k+2)m^{2}+\frac{1}{4}(3k-4)m+\frac{1}{2}\\
&<& {n-(k-1)m-2\choose 2}+(k-1)m^{2}+(k+1)m+3
\end{eqnarray*}
for $n\geq (7k-2)m+4$. This contradicts the assumption that $e(H)\geq {n-(k-1)m-2\choose 2}+(k-1)m^{2}+(k+1)m+3.$

\vspace{1.5mm}
\noindent\textbf{Case 2.} $\frac{n}{3}+(k-1)m+1 <r\leq n-(k-1)m-2.$
\vspace{1.5mm}

Note that $$e(H[C])={r\choose 2} ~~\mbox{and}~~ e(C,V(F))+e(F)\leq \sum_{u\in V(F)}d_{H}(u).$$

\begin{claim}\label{cla2}
{\rm $d_{H}(u)\leq n-(k-2)m-r-1$ for each $u\in V(F)$.}
\end{claim}

\begin{proof}
Suppose to the contrary that $d_{H}(u)\geq n-(k-2)m-r$ for some vertex $u\in V(F)$.
Then $d_{H}(u)+d_{H}(v)\geq (n-km+2m-r)+(r-1)=n-(k-2)m-1$ for each $v\in C$. Note that $H$ is $(n-km+2m-1)$-closed.
Then the vertex $u$ is adjacent to every vertex of $C$. This implies that $C\cup\{u\}$ is a larger clique, a contradiction.
\end{proof}

By Claim \ref{cla2}, we have
\begin{eqnarray*}
e(H)&=&e(H[C])+e(C,V(F))+e(F)\\
&\leq&{r\choose 2}+\sum_{u\in V(F)}d_{H}(u)\\
&\leq&{r\choose 2}+(n-r)(n-km+2m-r-1)\\
&=&\frac{3}{2}r^{2}-(2n-km+2m-\frac{1}{2})r+n^{2}-(km-2m+1)n.
\end{eqnarray*}
Let $f(r)=\frac{3}{2}r^{2}-(2n-km+2m-\frac{1}{2})r+n^{2}-(km-2m+1)n.$ Obviously, $f(r)$ is a concave function on $r.$
Since $\lfloor\frac{n}{3}\rfloor+(k-1)m+2 \leq r\leq n-(k-1)m-2,$ we have
\begin{eqnarray*}
e(H)&\leq& \max \left\{f(\lfloor\frac{n}{3}\rfloor+(k-1)m+2 ),f(n-(k-1)m-2)\right\}\\
&=&{n-(k-1)m-2\choose 2}+(k-1)m^{2}+(k+1)m+2<e(H)
\end{eqnarray*}
for $n\geq (3k-3)m+6$, a contradiction.

By Cases 1 and 2, we know that $r\geq n-(k-1)m-1,$ so $\omega(H)\geq r\geq n-(k-1)m-1$. The proof is completed.
\end{proof}

Using Lemma \ref{le5}, we next prove a technical sufficient condition in terms of $e(G)$ to assure that
an $m$-connected graph $G$ has a spanning $k$-tree.
Recall that $I_{n}$ denotes the complement of $K_{n}$.

\begin{lem}\label{le6}
Let $G$ be an $m$-connected graph of order $n\geq (7k-2)m+4$, where $m$ and $k$ are integers with $m\geq 1$ and $k\geq 2$.
If $$e(G)\geq {n-(k-1)m-2\choose 2}+(k-1)m^{2}+(k+1)m+3,$$ then $G$ has a spanning $k$-tree unless
$C_{n-(k-2)m-1}(G)\cong K_{m}\vee(K_{n-km-1}+I_{km-m+1})$.
\end{lem}

\begin{proof}
Suppose that $G$ has no spanning $k$-tree, where $n\geq (7k-2)m+4$, $m\geq 1$ and $k\geq2$.
Let $H=C_{n-(k-2)m-1}(G)$. It suffices to prove that $H\cong K_{m}\vee(K_{n-km-1}+I_{km-m+1}).$

By Theorem \ref{th5}, $H$ has no spanning $k$-tree.
Note that $e(G)\geq {n-(k-1)m-2\choose 2}+(k-1)m^{2}+(k+1)m+3$ and $G\subseteq H$. Then $e(H)\geq {n-(k-1)m-2\choose 2}+(k-1)m^{2}+(k+1)m+3$.
By Lemma \ref{le5}, we have $\omega(H)\geq n-(k-1)m-1$.

Next we characterize the structure of $H$. First we prove the following claim.

\begin{claim}\label{cla3}
{\rm $\omega(H)=n-(k-1)m-1$.}
\end{claim}

\begin{proof}
Note that $\omega(H)\geq n-(k-1)m-1$. It suffices to show that $\omega(H)\leq n-(k-1)m-1$.
Suppose to the contrary that $\omega(H)\geq n-(k-1)m$.
Let $C'$ be an $(n-km+m)$-clique of $H$ and $F'$ be a subgraph of $H$ induced by $V(H)\backslash C'$.
Since $m\geq1$ and $k\geq2$, we have $|V(F')|=(k-1)m>0$. Note that $G$ is an $m$-connected graph.
Then $H$ is also $m$-connected, so $\delta(H)\geq m$.
Then for each $u\in V(C')$ and each $v\in V(F')$, we have
$$d_{H}(u)+d_{H}(v)\geq (n-km+m-1)+m=n-(k-2)m-1.$$
Note that $H$ is an $(n-km+2m-1)$-closed graph.
Then each vertex $v$ of $F'$ is adjacent to every vertex of $C'$. It follows that $K_{n-km+m}\vee I_{km-m}$ is a spanning subgraph of $H$.
Next we construct a spanning $k$-tree of $K_{n-km+m}\vee I_{km-m}$.
Note that $n\geq (7k-2)m+4$. Then $n-km+m\geq (7k-2)m+4-km+m=(6k-1)m+4>km-m$,
so $K_{n-km+m}\vee I_{km-m}$ has a Hamilton path.
Then $K_{n-km+m}\vee I_{km-m}$ has a spanning $k$-tree for $k\geq 2$.
It follows that $H$ has a spanning $k$-tree, a contradiction. Therefore, $\omega(H)=n-(k-1)m-1$.
\end{proof}

Let $C$ be a maximum clique of $H$ and $F$ be a subgraph of $H$ induced by $V(H)\backslash C.$
By Claim \ref{cla3}, we know that $C\cong K_{n-(k-1)m-1}$.
Let $V(C)= \{u_{1} ,u_{2}, \ldots, u_{n-(k-1)m-1}\}$ and $V(F)= \{v_{1} ,v_{2}, \ldots, v_{(k-1)m+1}\}$.

\begin{claim}\label{cla5}
{\rm $d_{H}(v_{i})=m$ for every $v_{i}\in V(F).$}
\end{claim}

\begin{proof}
Recall that $\delta(H)\geq m$. Then for every $v_{i}\in V(F),$ we always have $d_{H}(v_{i})\geq m$.
Suppose that there exists a vertex $v_{i}\in V(F)$ with $d_{H}(v_{i})\geq m+1.$ Then
$$d_{H}(u)+d_{H}(v_{i})\geq (n-km+m-2)+(m+1)=n-(k-2)m-1$$
for each $u\in V(C)$.
Note that $H$ is an $(n-km+2m-1)$-closed graph. Then the vertex $v_{i}$ is adjacent to every vertex of $C$.
This implies that $C\cup\{v_{i}\}$ is a larger clique, a contradiction. Hence $d_{H}(v_{i})= m$ for every $v_{i}\in V(F).$
\end{proof}

\begin{claim}\label{cla6}
{\rm $N_{H}(v_{i})\cap C=N_{H}(v_{j})\cap C,$ where $i\neq j$.}
\end{claim}

\begin{proof}
Take arbitrarily $u_{i}\in N_{H}(v_{i})\cap C$. Then for each $v_{j}\in V(F)$, we have
$$d_{H}(u_{i})+d_{H}(v_{j})\geq (n-km+m-1)+m=n-(k-2)m-1,$$
where $i\neq j$. Note that $H$ is an $(n-km+2m-1)$-closed graph. Then $u_{i}$ is adjacent to every vertex of $F$.
It follows that $u_{i}\in N_{H}(v_{j})\cap C$, so $N_{H}(v_{i})\cap C\subseteq N_{H}(v_{j})\cap C$.
Similarly, we can prove that $N_{H}(v_{j})\cap C\subseteq N_{H}(v_{i})\cap C$.
Hence $N_{H}(v_{i})\cap C=N_{H}(v_{j})\cap C$, where $i\neq j$.
\end{proof}

\begin{claim}\label{cla4}
{\rm $V(F)$ is an independent set.}
\end{claim}

\begin{proof}
Recall that $H$ is $m$-connected.
By Claims \ref{cla5} and \ref{cla6}, $1 \leq |N_{H}(v)\cap C| \leq m$ for every $v\in V(F)$.
We can claim that $|N_{H}(v)\cap C|= m$. In fact, if $1 \leq |N_{H}(v)\cap C| \leq m-1$, then
$H[V(H)\setminus (N_{H}(v)\cap C)]$ is not connected,
which contradicts the $m$-connectivity of $H$.
Then $d_C(v)=|N_{H}(v)\cap C|=m$ for every $v\in V(F)$. Combining Claim \ref{cla5}, we know that $V(F)$ is an independent set.
\end{proof}

\begin{figure}
\centering
\includegraphics[width=0.4\textwidth]{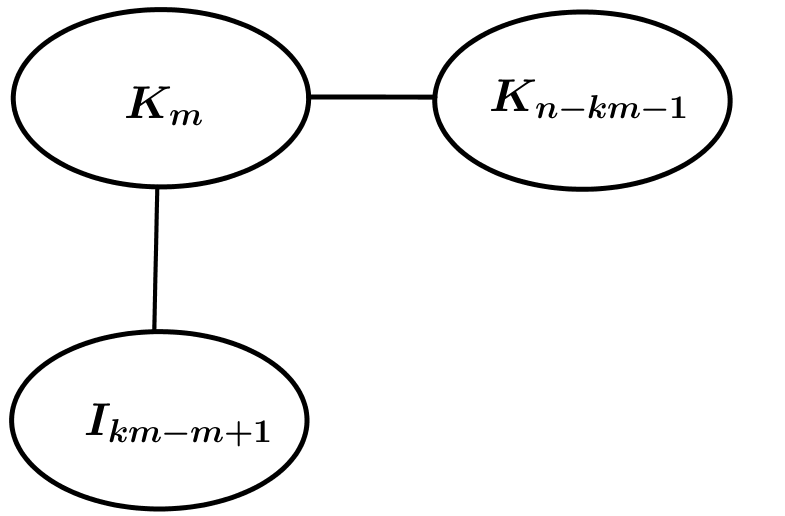}\\
\caption{Graph $K_{m}\vee(K_{n-km-1}+I_{km-m+1})$.
}\label{fig2}
\end{figure}

By the above claims, we know that $H\cong K_{m}\vee(K_{n-km-1}+I_{km-m+1})$ (see Fig. \ref{fig2}).
Note that the vertices of $I_{km-m+1}$ are only adjacent to $m$ vertices of $K_{n-(k-1)m-1}$.
Then the maximum degree of each spanning tree of $K_{m}\vee(K_{n-km-1}+I_{km-m+1})$ is at least $k+1$.
This implies that $K_{m}\vee(K_{n-km-1}+I_{km-m+1})$ has no spanning $k$-tree.
Note that $e(H)={n-(k-1)m-1\choose 2}+(k-1)m^{2}+m$ and  $n\geq (7k-2)m+4>(2k-1)m+5$. Then we have
$$e(H)\geq {n-(k-1)m-2\choose 2}+(k-1)m^{2}+(k+1)m+3.$$
Therefore, $H=C_{n-(k-2)m-1}(G)\cong K_{m}\vee(K_{n-km-1}+I_{km-m+1}),$ as desired.
\end{proof}

Let $A=(a_{ij})$ and $B=(b_{ij})$ be two $n\times n$ matrices.
Define $A\leq B$ if $a_{ij}\leq b_{ij}$ for all $i$ and $j$, and define $A< B$ if $A\leq B$ and $A\neq B$.

\begin{lem}[Berman and Plemmons \cite{Berman1979}, Horn and Johnson \cite{Horn1986}]\label{le7}
Let $A=(a_{ij})$ and $B=(b_{ij})$ be two $n\times n$ matrices with the spectral radii $\lambda(A)$ and $\lambda(B)$, respectively.
If $0\leq A\leq B$, then $\lambda(A)\leq \lambda(B)$.
Furthermore, if $B$ is irreducible and $0\leq A < B$, then $\lambda(A)<\lambda(B)$.
\end{lem}

\begin{lem}[Hong \cite{Hong1988}]\label{le8}
Let $G$ be a connected graph with $n$ vertices. Then
$$\rho(G)\leq \sqrt{2e(G)-n+1}.$$
\end{lem}

Now, we are ready to give the proof of Theorem \ref{main2}.

\medskip
\noindent  \textbf{Proof of Theorem \ref{main2}.}
Let $G$ be an $m$-connected graph of order $n\geq {\rm max}\{(7k-2)m+4, (k-1)m^{2}+\frac{1}{2}(3k+1)m+\frac{9}{2}\}$,
where $m\geq 1$ and $k\geq2.$ Suppose to the contrary that $G$ has no spanning $k$-tree.
Note that $K_{n-(k-1)m-1}$ is a proper subgraph of $K_{m}\vee(K_{n-km-1}+I_{km-m+1})$.
By Lemma \ref{le7}, we know that $\rho(K_{m}\vee(K_{n-km-1}+I_{km-m+1}))>\rho(K_{n-(k-1)m-1})=n-(k-1)m-2$.
By Lemma \ref{le8} and the assumption of Theorem \ref{main2}, we have
$$n-(k-1)m-2<\rho(K_{m}\vee(K_{n-km-1}+I_{km-m+1})) \leq \rho(G)\leq \sqrt{2e(G)-n+1}.$$
By simple calculation, we obtain
\begin{eqnarray*}
e(G)&>&\frac{(n-km+m-2)^{2}+n-1}{2}\\
&\geq& {n-(k-1)m-2\choose 2}+(k-1)m^{2}+(k+1)m+3
\end{eqnarray*}
for $n\geq (k-1)m^{2}+\frac{1}{2}(3k+1)m+\frac{9}{2}$. Let $H=C_{n-(k-2)m-1}(G)$.
By Lemma \ref{le6}, we have $H\cong K_{m}\vee(K_{n-km-1}+I_{km-m+1}).$ Note that $G\subseteq H.$
Then $$\rho(G)\leq\rho(H)=\rho(K_{m}\vee(K_{n-km-1}+I_{km-m+1})).$$
Combining the assumption that $\rho(G)\geq \rho(K_{m}\vee(K_{n-km-1}+I_{km-m+1}))$, we have $G\cong H.$
From the end of the proof in Lemma \ref{le6}, we know that $K_{m}\vee(K_{n-km-1}+I_{km-m+1})$ has no spanning $k$-tree.
Hence $G\cong K_{m}\vee(K_{n-km-1}+I_{km-m+1}).$
\hspace*{\fill}$\Box$

\section{Proof of Theorem \ref{main3}}

By the Perron-Frobenius Theorem, there exists a unique positive unit eigenvector corresponding to $\rho(G)$,
which is called the {\it Perron vector} of $G$, and $\rho(G)$ is always a positive number (unless $G$ is an empty graph).

\medskip
\noindent  \textbf{Proof of Theorem \ref{main3}.}
Let $G$ be a connected graph of order $n\geq 3(k+2)\delta+2$ with minimum degree $\delta\geq 1$ and $k\geq 1$.
Suppose to the contrary that $G$ has no spanning tree with leaf degree at most $k$.
By Theorem \ref{th6}, there exists some nonempty subset $S\subseteq V(G)$ such that $i(G-S)\geq (k+1)|S|.$
Let $|S|=s $. It is easy to see that $s\geq \delta$.
Note that $G$ is a spanning subgraph of $G'=K_{s}\vee(K_{n-ks-2s}+I_{ks+s})$ (see Fig. \ref{fig3}).
By Lemma \ref{le7}, we have $\rho(G)\leq \rho(G')$.
We divide the proof into the following two cases.
\begin{figure}
\centering
\includegraphics[width=0.4\textwidth]{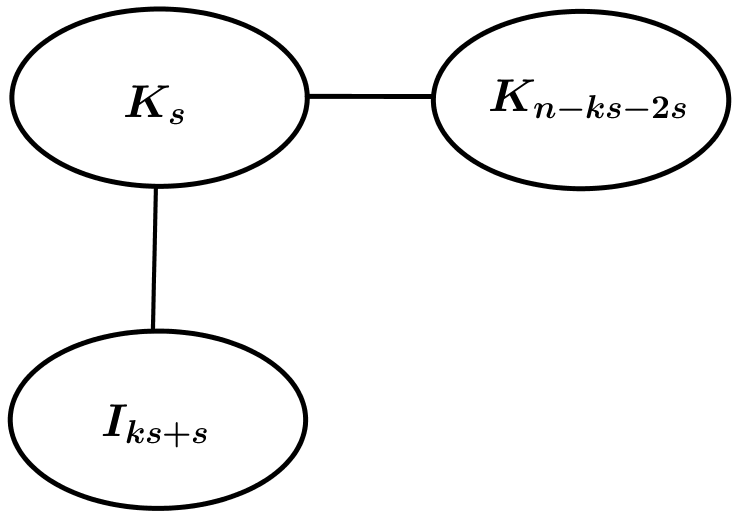}\\
\caption{Graph $K_{s}\vee(K_{n-ks-2s}+I_{ks+s}).$
}\label{fig3}
\end{figure}

\vspace{1.5mm}
\noindent\textbf{Case 1.}  $s=\delta.$
\vspace{1mm}

Then $G'\cong K_{\delta}\vee(K_{n-k\delta-2\delta}+I_{k\delta+\delta}).$ From the above, we know that
$$\rho(G)\leq \rho(K_{\delta}\vee(K_{n-k\delta-2\delta}+ I_{k\delta+\delta})).$$
By the assumption that $\rho(G)\geq \rho(K_{\delta}\vee(K_{n-k\delta-2\delta}+ I_{k\delta+\delta})),$ we have $G\cong K_{\delta}\vee(K_{n-k\delta-2\delta}+I_{k\delta+\delta}).$
Note that the vertices of $I_{k\delta+\delta}$ are only adjacent to $\delta$ vertices of $K_{n-k\delta-\delta}$.
Then the leaf degree of any spanning tree of $K_{\delta}\vee(K_{n-k\delta-2\delta}+I_{k\delta+\delta})$ is at least $k+1$.
This implies that $K_{\delta}\vee(K_{n-k\delta-2\delta}+I_{k\delta+\delta})$ has no spanning tree with leaf degree at most $k$.
Hence $G\cong K_{\delta}\vee(K_{n-k\delta-2\delta}+I_{k\delta+\delta})$.

\vspace{1.5mm}
\noindent\textbf{Case 2.}  $s\geq \delta+1.$
\vspace{1mm}

Recall that $G'=K_{s}\vee(K_{n-ks-2s}+I_{ks+s})$. Note that $n\geq (k+2)s$. Then $\delta+1 \leq s \leq \frac{n}{k+2}$.
Since $G'$ is not a complete graph, $\rho'<n-1.$
Let $\mathbf{x}$ be the Perron vector of $A(G')$, and let  $\rho'=\rho(G')$.
By symmetry, $\mathbf{x}$ takes the same value $(x_{1},x_{2},x_{3})$ on the partition ($V(K_{s}), V(K_{n-ks-2s}), V(I_{ks+s}))$.
By $A(G')\mathbf{x}=\rho' \mathbf{x}$, we have $\rho'x_3=sx_1.$ Since $\rho'>0$,
\begin{equation}\label{eq1}
x_3=\frac{sx_1}{\rho'}.
\end{equation}

Let $V(K_{s})=\{u_1, u_2, \ldots, u_s\}$, $V(K_{n-ks-2s})=\{v_1, v_2, \ldots, v_{n-ks-2s}\}$ and
$V(I_{ks+s})=\{w_1, w_2, \ldots, w_{ks+s}\}.$
Let $G^*=K_{\delta}\vee (K_{n-k\delta-2\delta}+I_{k\delta+\delta})$, and let
$E_1=\{v_{i}w_{j}|1\leq i\leq n-(k+2)s, (k+1)\delta+1\leq j\leq (k+1)s\}\cup \{w_{i}w_{j}|(k+1)\delta+1\leq i\leq (k+1)s-1, i+1\leq j\leq (k+1)s\}$ and $E_2=\{u_{i}w_{j}|\delta+1\leq i\leq s, 1\leq j\leq (k+1)\delta\}$.
Obviously, $G^*\cong G'+E_1-E_2.$
Let $\mathbf{y}$ be the Perron vector of $A(G^{*})$, and let $\rho^*=\rho(G^*)$.
By symmetry, $\mathbf{y}$ takes the same values $(y_{1},y_{2},y_{3})$ on the partition
$(V(K_{\delta}), V(K_{n-k\delta-2\delta}), V(I_{k\delta+\delta}))$.
By $A(G^{*})\mathbf{y}=\rho^*\mathbf{y}$, we have
\begin{gather}\nonumber
\rho^*y_2=\delta y_1+(n-k\delta-2\delta-1)y_2  ~~\mbox{and}~~\rho^*y_3=\delta y_1.
\end{gather}
Note that $K_{n-k\delta-2\delta}$ is a proper subgraph of $G^{*}$. By Lemma \ref{le7}, we have $\rho^*>\rho(K_{n-k\delta-2\delta})=n-k\delta-2\delta-1$.
By simple calculation, we have
\begin{equation}\label{eq2}
y_2=\frac{\rho^*y_3}{\rho^*-(n-k\delta-2\delta-1)}.
\end{equation}

\begin{claim}\label{cla7}
{\rm  $\rho^*>\rho'$.}
\end{claim}

\begin{proof}
Suppose to the contrary that $\rho'\geq \rho^*$. Note that $n\geq 3(k+2)\delta+2$.
Combining $n\geq (k+2)s$, $\delta+1\leq s\leq \frac{n}{k+2}$, (\ref{eq1}), and (\ref{eq2}), we have
\begin{eqnarray*}
&&\mathbf{y}^{T}(\rho^*-\rho')\mathbf{x}=\mathbf{y}^{T}(A(G^*)-A(G'))\mathbf{x}\\
&=&\sum_{i=1}^{n-(k+2)s}\sum_{j=(k+1)\delta+1}^{(k+1)s}(x_{v_{i}}y_{w_j}+x_{w_j}y_{v_i})
+\sum_{i=(k+1)\delta+1}^{(k+1)s-1}\sum_{j=i+1}^{(k+1)s}(x_{w_i}y_{w_j}+x_{w_j}y_{w_i})
-\sum_{i=\delta+1}^{s}\sum_{j=1}^{(k+1)\delta}(x_{u_i}y_{w_j}+x_{w_j}y_{u_i})\\
&=&(k+1)(s-\delta)\left((n-ks-2s)(x_2y_2+x_3y_2)+((k+1)(s-\delta)-1)x_3y_2-\delta(x_1y_3+x_3y_2)\right)\\
&=&(k+1)(s-\delta)\left((n-k\delta-2\delta-s-1)x_3y_2+(n-ks-2s)x_2y_2-\delta x_1y_3\right)\\
&\geq&(k+1)(s-\delta)\left((n-k\delta-2\delta-s-1)x_3y_2-\delta x_1y_3\right)\\
&=&(k+1)(s-\delta)x_1y_3\left(\frac{\rho^*s(n-k\delta-2\delta-s-1)}{\rho'(\rho^*-(n-k\delta-2\delta-1))}-\delta\right)\\
&=&\frac{\rho^*\delta(k+1)(s-\delta)x_1y_3}{\rho'(\rho^*-(n-k\delta-2\delta-1))}\left(\frac{s}{\delta}(n-k\delta-2\delta-s-1)
-\rho'+\frac{\rho'}{\rho^*}(n-k\delta-2\delta-1)\right)\\
&>&\frac{\rho^*\delta(k+1)(s-\delta)x_1y_3}{\rho'(\rho^*-(n-k\delta-2\delta-1))}\left((n-k\delta-2\delta-s-1)-\rho'+(n-k\delta-2\delta-1)\right)\\
&=&\frac{\rho^*\delta(k+1)(s-\delta)x_1y_3}{\rho'(\rho^*-(n-k\delta-2\delta-1))}\left(2n-2(k+2)\delta-s-2-\rho'\right).
\end{eqnarray*}
Let $f(n)=2n-2(k+2)\delta-s-2-\rho'$. Note that $s\leq\frac{n}{k+2}$, $k\geq1$, $n\geq 3(k+2)\delta+2$ and $\rho'<n-1$. Then
\begin{eqnarray*}
f(n)&=&2n-2(k+2)\delta-s-2-\rho'\\
&\geq&2n-2(k+2)\delta-\frac{n}{k+2}-2-\rho'\\
&\geq&2n-2(k+2)\delta-\frac{n}{3}-2-\rho'\\
&=&\frac{5n}{3}-2(k+2)\delta-2-\rho'\\
&\geq&n+\frac{2}{3}(3(k+2)\delta+2)-2(k+2)\delta-2-\rho'\\
&=&n-\rho'-\frac{2}{3}>0.
\end{eqnarray*}
From the above analysis, we have $\mathbf{y}^{T}(\rho^*-\rho')\mathbf{x}>0$. It follows that $\rho^*>\rho',$
which contradicts the assumption. Hence $\rho^*>\rho'.$
\end{proof}

By Claim \ref{cla7}, we have $\rho(G)\leq\rho(G')<\rho(G^*),$ which contradicts the hypothesis of this theorem, as desired. \hspace*{\fill}$\Box$

\vspace{5mm}
\noindent
{\bf Declaration of competing interest}
\vspace{3mm}

The authors declare that they have no known competing financial interests or personal relationships that could have appeared to influence the work reported in this paper.



\end{document}